\documentclass[12pt]{amsart}
\usepackage{amsfonts}
\usepackage{amssymb}
\usepackage{color}
\usepackage{dsfont}
\usepackage{hyperref}
\numberwithin{equation}{section}
\theoremstyle{plain}
 \newtheorem{theorem}{Theorem}[section]
 \newtheorem{lemma}[theorem]{Lemma}
 \newtheorem{corollary}[theorem]{Corollary}
 \newtheorem{proposition}[theorem]{Proposition}

\theoremstyle{definition}
 \newtheorem{definition}[theorem]{Definition}
 \newtheorem{example}[theorem]{Example}
 \newtheorem{remark}[theorem]{Remark}

\newcommand{\bN}{\mathbb{N}}
\newcommand{\bZ}{\mathbb{Z}}
\newcommand{\bR}{\mathbb{R}}

\newcommand{\bC}{\mathbb{C}}

\newcommand{\bE}{\mathbb{E}}

\newcommand{\cov}{\mbox{\rm cov}\,}

\newcommand{\one}{\mathbf{1}}

\allowdisplaybreaks

\begin{document}
% \author{Alexander Lindner \and Lei Pan \and Ken-iti Sato}
% \title{On quasi-infinitely divisible distributions}
% \maketitle
%\begin{flushright}
%{\tiny To Alex on \today}
%\end{flushright}

\vspace{5mm}
\begin{center}
{\bf
{\large
Central Limit Theorems for Moving Average Random Fields with Non-Random and Random Sampling}}

\vspace{5mm}

David Berger\\
\end{center}
\bigskip

{\bf Abstract.}
For a L\'evy basis $L$ on $\bR^d$ and a suitable kernel function $f:\bR^d \to \bR$, consider the continuous spatial moving average field $X=(X_t)_{t\in \bR^d}$ defined by $X_t = \int_{\bR^d} f(t-s) \, dL(s)$. Based on observations on finite subsets $\Gamma_n$ of $\bZ^d$, we obtain central limit theorems for the sample mean and the sample autocovariance function of this process. We allow sequences $(\Gamma_n)$ of deterministic subsets of $\bZ^d$ and of random subsets of $\bZ^d$. The results generalise existing results for time indexed stochastic processes (i.e. $d=1$) to random fields with arbitrary spatial dimension $d$, and additionally allow for random sampling. The results are  applied to obtain a consistent and asymptotically normal estimator of $\mu>0$ in the stochastic partial differential equation $(\mu - \Delta) X = dL$ in dimension 3, where $L$ is L\'evy noise.

\section{Introduction}
Many statistical models with more than one spatial dimension are described by a linear stochastic partial differential equation with some additive noise, which means that we have a random field $X$ on $\bR^d$ satisfying
\begin{align}\label{eqspde}
\mathcal{L}(\mu)X=dL,
\end{align}
where $\mathcal{L}(\mu)$ is a linear partial differential operator depending on some parameter $\mu$ and $dL$ denotes some noise, for example Gaussian or stable noise. If $\mathcal{L}(\mu)$ has an integrable fundamental solution $G_{\mu}$ the mild solution of $(\ref{eqspde})$ can be written as
\begin{align}\label{eq1000}
X_t=\int\limits_{\bR^d}G_{\mu}(t-s)dL(s),
\end{align}
where $dL$ denotes the additive noise, see for example  [\ref{Barndorff}], [\ref{Berger}], [\ref{Pham}] and [\ref{Walsh}]. The solution $(\ref{eq1000})$ is a so called continuous moving average random field.\\
The additive noise $dL$ studied in this paper will be a L\'{e}vy white noise, where the Gaussian white noise and stable noise are included. A detailed study of L\'{e}vy white noise can be found in [\ref{Fageot1}], where it is also shown that a L\'{e}vy white noise defines a L\'{e}vy basis in the sense of Rajput and Rosinski [\ref{Rajput}].
 %In many cases one is interested in estimating the parameter $\mu$ of the equation $(\ref{eqspde})$. If we know how the fundamental solution $G_{\mu}$ depends on the parameter $\mu$ it is possible to give moment estimators for $\mu$, which is our starting point. We will prove several central limit theorems the sample mean and the sample autocovariance of random fields in the form $(\ref{eq1000})$. \\
Random fields of the form
\begin{align}\label{eq1000A}
X_t=\int\limits_{\bR^d} f(t-s)\, dL(s),
\end{align}
with a suitable kernel function $f:\bR^d\to \bR$ and a L\'evy basis $L$ on $\bR^d$ (as in \eqref{eq1000} with $f=G_\mu$) can be seen as a continuous 
and spatial extension of the discrete time moving average processes $Z=(Z_t)_{t\in \bZ},$ defined by
\begin{align} \label{eq1000B}
Z_t=\sum\limits_{k\in\bZ} a_{t-k}W_k,
\end{align}
where $(W_k)_{k\in \bZ}$ is an independent and identically distributed sequence and $a_k$, $k\in \bZ$, are real coefficients.

In many cases one is interested in estimating the parameter $\mu$ of the equation $(\ref{eqspde})$. If we know how the fundamental solution $G_{\mu}$ depends on the parameter $\mu$, it is sometimes possible to give moment estimators for $\mu$. Of particular interest are estimators of the mean $\bE (X_t)$ and the autocovariance $\cov (X_t , X_{t+h})$ for $t,h\in \bR^d$. In most applications 
only discrete spatial data is available, for example observations based on a finite subset $\Gamma_n$ of the regular grid $\bZ^d$. A natural estimator for $\bE X_t$ is then the sample mean $\frac{1}{|\Gamma_n|} \sum_{s\in \Gamma_n} X_s$, while a natural estimator for the autocovariance $\cov (X_t , X_{t+h})$ is the (adjusted) sample autocovariance \begin{align} \label{eq1000C}
\gamma_n^* (h) := \frac{1}{|\Gamma_n|} \sum_{s\in \Gamma_n} X_s X_{s+h}, \quad h\in \bZ^d
\end{align}
(assuming that the L\'evy basis and hence $X$ have mean zero and that for each $s\in \Gamma_n$, both $X_s$ and $X_{s+h}$ are observed). Motivated by this, in this paper we will provide central limit theorems for the sample mean and sample autocovariance function as defined in \eqref{eq1000C} for continuous spatial moving average random fields as defined in \eqref{eq1000A} (equivalently, \eqref{eq1000}), when  the kernel function $f$ decays sufficiently fast and the L\'evy basis has finite variance or finite fourth moment and mean zero, respectively. The sampling sequence $(\Gamma_n)_{n\in \bN}$ will be  a nested sequence of finite subsets of $\bZ^d$ satisfying $|\Gamma_n|\to \infty$ and some extra conditions, and it will be either a sequence of deterministic subsets (referred to as non-random sampling) or a sequence of random subsets (referred to as random-sampling), more precisely of the form $\Gamma_n = \{ t \in [-n,n)^d \cap \bZ^d | Y_t =1\}$, where $(Y_t)_{t\in \bZ^d}$ is a $\{0,1\}-$valued stationary ergodic random field on $\bZ^d$. In the case of non-random sampling, we will need slightly higher moment conditions on the L\'evy basis.

Central limit theorems for the sample mean and the sample autocovariance of \eqref{eq1000B} are classic and can be found e.g.~in Chapter 7 of the book \cite{Brockwell} (for $d=1$). On the other hand, central limit theorems for L\'evy driven moving average processes based on discrete low-frequency observations have only recently attracted attention, and this also only in dimension $d=1$, i.e. for continuous time series and not spatial data. In \cite{Lindner}, the asymptotics of the sample mean and sample autocovariance are studied when $f$ decays sufficiently fast and $L$ has finite second or fourth moment, respectively. \cite{Spangenberg} studies the situation when $f$ decays slowly leading to a long-memory process $X$, while \cite{Drapatz} considers the heavy tailed situation when the L\'evy process $L$ is in the domain of attraction of a stable non-normal distribution, and in \cite{Brandes} the case of random sampling when the process $X$ is sampled at a renewal sequence is treated. Observe that all these results are in dimension $d=1$ only. The results of this paper can be seen as a generalization of the results of \cite{Lindner}, who have $d=1$  and $\Gamma_n =\{1,2, \ldots, n\}$, to arbitrary spatial dimensions $d\in \bN$ and more general sets $\Gamma_n$, and additionally allowing random sampling as described above.

The paper is organized as follows. In the next section, we fix notation and recall the notion of L\'evy bases. Then, in Section \ref{S3}, we state the main results of the present paper. 
These are central limit theorems for the sample mean as described above for non-random and random sampling (Theorems \ref{marc} and \ref{extra}, respectively), and central limit theorems for the sample autocovariance as described above for non-random and random sampling (Theorems \ref{sophie3} and \ref{extra1}, respectively). In Section \ref{S4} we apply the results to a random field given as a solution as in \eqref{eqspde}, more specifically, we consider the stochastic partial differential equation
$$(\mu - \Delta) X = dL$$
in dimension $d=3$, where $\Delta$ denotes the Laplace operator, and obtain a  consistent and asymptotically normal estimator of $\mu>0$ based on the sample mean. Finally, Sections \ref{sectionproof1} and \ref{sectionproof2} contain the proofs of the main theorems for the sample mean and the sample autocovariance,  respectively.

\section{Notation and Preliminaries} \label{S2}
To fix notation, by a distribution on $\bR$ we mean a probability measure on $(\bR,\mathcal{B}(\bR))$ with $\mathcal{B}(\bR)$ being the Borel $\sigma-$algebra on $\bR$. By a measure on $\bR^d$, $d$ a natural number, we always mean a positive measure on $(\bR^d,\mathcal{B}(\bR^d))$. The set $\mathcal{B}_b(\bR^d)$ is the set of all bounded Borel measurable sets. The Dirac measure at a point $b\in \bR$ will be denoted by $\delta_{b}$, the Gaussian distribution with mean $a\in\bR$ and variance $b\ge 0$ by $N(a,b)$ and the Lebesgue measure by $\lambda^d$ on $\bR^d$. If a random vector $X$ has law $\mathcal{L}$ we write $X\sim \mathcal{L}$. Weak convergence of measures will be denoted by "$\stackrel{d}{\to}$". We write $\bN=\{1,2,\dotso\}$, $\bN_0=\bN\cup \{0\}$ and $\bZ,\,\bR$ for the set of integers and real numbers  respectively. The indicator function of a set $A\subset \bR$ is denoted by $\one_A$. By $L^p(\bR^d, A)$ for $1\le p<\infty$ and $A\subset \bC$ we denote the set of all Borel-measurable functions $f:\bR^d \to A$ such that $\int_{\bR^d} |f(x)|^p\,\lambda^d(dx)<\infty$. If $A=\bR$ we simply write $L^p(\bR^d)$. For two different sets $A,B\subset\bR^d$, we denote  $dist(A,B):=\inf\{\|x-y\|:x\in A\textrm{ and }y\in B\}$, where $\|\cdot\|$ is the euclidean norm. We write `a.e.' to denote almost everywhere and `a.s.'  to denote almost surely. $|A|$ denotes the number of elements of the set $A$.\\
We are interested in integrals of the form $\int_{\bR^d} f(u)\,dL(u)$, where $dL$ denotes the integration over a L\'{e}vy basis. A L\'{e}vy basis can be understood in the following way:
\begin{definition}[see {[\ref{Rajput}, p. 455]}]\label{randommeasures}
A \emph{L\'{e}vy basis} is family $(L(A))_{A\in\mathcal{B}_b({\bR^d})}$ of real valued random variables such that
\begin{itemize}
\item[i)]$L(\bigcup_{n=0}^\infty A_n)=\sum_{n=0}^\infty L(A_n)$ $a.s.$ for pairwise disjoint sets $(A_n)_{n\in\bN_0}\subset \mathcal{B}_b(\bR^d)$ with $\bigcup_{n\in\bN_0}A_n\in \mathcal{B}_b(\bR^d)$,
\item[ii)] $L(A_i)$ are independent for pairwise disjoint sets $A_1,\dotso,A_n\in \mathcal{B}_b(\bR^d)$ for every $n\in\bN$,
\item[iii)] there exist $a\in [0,\infty)$, $\gamma \in \bR$ and a L\'{e}vy measure $\nu$ on $\bR$ (i.e. a measure $\nu$ on $\bR$ such that $\nu(\{0\})=0$ and $\int\limits_{\bR} \min\{1,x^2\}\nu(dx)<\infty$) such that
\begin{align*}
\bE e^{iz L(A)}=\exp\left(\psi(z)\lambda^d(A)\right)
\end{align*}
for every $A\in \mathcal{B}_b(\bR^d)$, where
\begin{align*}
\psi(z):=i\gamma z-\frac{1}{2}az^2+\int\limits_{\bR} (e^{ixz}-1-ixz\one_{[-1,1]}(x))\nu(dx),\quad z\in \bR.
\end{align*}
The triplet $(a,\nu ,\gamma)$ is called the \emph{characteristic triplet} of $L$ and $\psi$ its \emph{characteristic exponent}. By the L\'{e}vy-Khintchine formula, $L(A)$ is then infinitely divisible. 
\end{itemize}
\end{definition}
It can be shown that the characteristic triplet is unique; conversely, to every $a\in [0,\infty)$, $\gamma \in \bR$ and L\'{e}vy measure $\nu$ there exists a L\'{e}vy basis with $(a,\nu,\gamma)$ as characteristic triplet.
It follows from the general theory of infinitely divisible distributions that for a L\'{e}vy basis $L$ with characteristic triplet $(a,\nu,\gamma)$ and $p\in [1,\infty)$, we have $\int\limits_{|x|>1}|x|^p \nu(dx)<\infty$ if and only if $\bE |L(A)|^p<\infty$ for some (equivalently, all) $A\in \mathcal{B}_b(\bR^d)$ with $\lambda^d(A)>0$. In that case,
\begin{align*}
\bE L(A)=\lambda^d(A) \bE L([0,1]^d).
\end{align*}
Integration of deterministic functions with respect to L\'{e}vy bases is described by \\Rajput and Rosinski [\ref{Rajput}]; in particular for \emph{simple functions} $f$ of the form $f=\sum\limits_{j=1}^n x_j\one_{A_j}$ with $x_j\in\bR$ and $A_j\in\mathcal{B}_b(\bR^d)$, the integral $\int\limits_{A}f(u)dL(u)$ for $A\in\mathcal{B}(\bR^d)$ is defined as $\sum\limits_{j=1}^n x_j L(A_j\cap A)$. A general Borel-measurable function $f:\bR^d\to \bR$ is called \emph{integrable with respect to $L$}, if there exists a sequence of simple functions $(f_n)_{n\in\bN}$ such that $f_n\to f$ $\lambda^d-$a.e. and such that $\int\limits_{A} f_n(u)dL(u)$ converges in probability as $n\to \infty$ for every $A\in \mathcal{B}(\bR^d)$, in which case this limit is denoted by $\int\limits_A f(u)dL(u)$, see [\ref{Rajput}, p.460]. Rajput and Rosinski also characterize integrability of functions. In particular, if $f\in L^1(\bR^d)\cap L^2(\bR^d)$ and $\bE L([0,1]^d)^2<\infty$, or if $f\in L^2(\bR^d)$, $\bE L([0,1]^d)^2<\infty$ and $\bE L([0,1]^d)=0$, then the integral $\int_{\bR^d} f(u)dL(u)$ is well-defined and satisfies $\bE \left(\int_{\bR^d} f(u)dL(u)\right)^2<\infty$. This follows by standard calculations. Moreover, for two such functions $f,g$ we have
\begin{align}\label{Ito}
\cov\left( \,\,\int\limits_{\bR^d}f(u)dL(u),\int\limits_{\bR^d}g(u)dL(u)\right)=\sigma^2 \int\limits_{\bR^d}f(u)g(u) \lambda^d(du),
\end{align}
where $\sigma^2 =\bE L([0,1]^d)^2$. For a stationary random field $X=(X_t)_{t\in\bR^d}$ with finite second moment we write $\gamma_X(t):=\cov(X_t,X_0)$.
\vspace{5mm}

\section{Main results}\label{sectionmain} \label{S3}
In this section, we formulate our main results. Our sampling grid will always be $\bZ^d$, but observe that every result can be extended to the sampling set $\Delta A\bZ^d=\{\Delta Av: v\in\bZ^d\}$, where $A$ is an orthogonal $d\times d$-matrix and $\Delta>0$, because the L\'{e}vy basis is invariant (in distribution) under orthogonal transformations and any scale transformation can be applied to the L\'{e}vy basis instead to the lattice by transporting the scaling parameter to the triplet $(a,\gamma,\nu)$.
Our sampling sets $\Gamma_n$ will then be subsets of $\bZ^d$. The process under consideration is given by $X_t=\int\limits_{\bR^d}f(t-s)dL(s)$, where $f:\bR^d\to \bR$ is integrable with respect to the L\'{e}vy basis $L$. By homogeneity of the L\'{e}vy basis, it is easy to see that $(X_t)_{t\in \bR^d}$ is a strictly stationary random field, meaning that its finite dimensional distributions are shift invariant.\\
The proof of Theorem \ref{marc} and Theorem \ref{extra} are in Section \ref{sectionproof1} and the proofs of Theorem \ref{sophie3} and Theorem \ref{extra1} in Section \ref{sectionproof2}.
\subsection{Central limit theorems for the sample mean}
In this and the next section, we give central limit theorems (CLTs) for the sample mean. 
\begin{theorem}\label{marc}
Let $L$ be a L\'{e}vy basis with $\bE (L([0,1]^d)^2<\infty$ and $f\in L^1(\bR^d)\cap L^2(\bR^d)$, and let 
\begin{align*}
X_t:= \int\limits_{\bR^d} f(t-u)dL(u),\quad t\in\bR^d.
\end{align*}
Let $(\Gamma_n)_{n\in \bN}$ be a sequence of finite subsets of $\bZ^d$ such that
\begin{itemize}
\item[a)] $\Gamma_n \subset \Gamma_{n+1}$ for every $n\in\bN$,
\item[b)] $|\Gamma_n|\to \infty$ as $n\to \infty$, and 
\item[c)] $a_l^n:=\frac{|\{(t,s)\in \Gamma_n \times \Gamma_n: t-s=l\}|}{|\Gamma_n|}$ converges as $n\to \infty$ to some $a_l$ for each $l\in \bZ^d$.
\end{itemize}
Assume that
\begin{align}\label{eq7766}
\sum\limits_{t\in \bZ^d} \sup\limits_{n\in \bN} a_t^n \int\limits_{\bR^d} |f(-u)f(t-u)|\lambda^d(du)<\infty.
\end{align}
Then
\begin{align*}
\sum\limits_{t\in \bZ^d} a_t|\cov(X_t,X_0)|<\infty,
\end{align*}
and
\begin{align*}
\frac{1}{\sqrt{|\Gamma_n|}} \sum\limits_{t\in \Gamma_n} \left(X_t- \bE L([0,1]^d) \int\limits_{\bR^d} f(u) \lambda^d(du)\right)\stackrel{d}{\to} N\left(0,\sum\limits_{t\in \bZ^d} a_t\cov(X_t,X_0)\right).
\end{align*}
\end{theorem}
\begin{remark}
From the definition of $a_l^n$ it is obvious that $0\le a_l^n\le 1$, hence necessarily also $a_l \in [0,1]$ for each $l\in \bN$.\\
A sufficient condition for (\ref{eq7766}) to  hold is hence that
\begin{align*}
\sum\limits_{t \in \bZ^d} \int\limits_{\bR^d} |f(-u)f(t-u)|\lambda^d(du)<\infty.
\end{align*}
Denoting
\begin{align*}
F(u):=\sum\limits_{t\in \bZ^d} |f(u+t)|,\quad u\in\bR^d,
\end{align*}
it is easy to see that $F$ is periodic and that 
\begin{align*}
\sum\limits_{t\in \bZ^d} \int\limits_{\bR^d} |f(-u)f(t-u)|\lambda^d(du)&=\int\limits_{\bR^d} |f(u)|F(u)\lambda^d(du)\\
&=\int\limits_{[0,1]^d} \sum\limits_{t\in \bZ^d}|f(u+t)|F(u)\lambda^d(du)\\
&=\int\limits_{[0,1]^d} F(u)^2 \lambda^d(du),
\end{align*}
so that $F\in L^2([0,1]^d)$ is a sufficient condition for $(\ref{eq7766})$ to hold.
Observe however that there are also other cases when (\ref{eq7766}) holds but  $F\notin L^2([0,1]^d)$. For example, when the sets $\Gamma_n$ are contained in some hyperplane of $\bR^d$, then many of the $a_l^n$ will be $0$.
\end{remark}

\begin{example}
Let $\Gamma_n=(-n,n]^d\cap \bZ^d$. Then it is clear that $a_l^n$ in Theorem \ref{marc} will converge to $1$ as $n\to \infty$ for each $l\in \bZ^d$. Sequences that satisfy $\lim_{n\to\infty}a_l^n=1$ for each $l$ are called \emph{F{\o}lner}. They play an important role in ergodic theorems in the theory of amenable groups, see [\ref{Lindenstrauss}]. 
\end{example}

Another example of sequences $(\Gamma_n)$ satisfying the assumptions of Theorem \ref{marc} can be obtained as realisations of certain random subsets, in which also the limits $a_l$ may be non-trivial (i.e. different from $0$ or $1$). This follows from the next lemma, where we use the concept of ergodicity on $\bZ^d$, see [\ref{Tempelman}, Definition 1.1, p. 52].
\begin{lemma}\label{lemmaergodicsampling}
Let $(Y_t)_{t\in\bZ^d}$ be a $\{0,1\}-$valued stationary ergodic random field such that $\bE Y_0\neq 0$ (i.e. $P(Y_0=0)<1$). We define
\begin{align*}
\Gamma_n:=\{ t\in [-n,n)^d\cap \bZ^d \,:\, Y_t=1\}.
\end{align*}
Then $(\Gamma_n)_{n\in\bN}$ satifies
\begin{align*}
\frac{\{(t,s)\in \Gamma_n\times \Gamma_n: t-s=l\}}{|\Gamma_n|}\to \frac{\bE Y_lY_0}{\bE Y_0}\qquad\textrm{a.s. for }n\to\infty.
\end{align*}
Especially, $(\Gamma_n)_{n\in\bN}$ satisfies almost surely the assumptions of Theorem \ref{marc}.
\end{lemma}
\begin{proof}
This is an easy application of the ergodic properties of $Z_t$. We write
\begin{align*}
&\frac{\{(t,s)\in \Gamma_n\times \Gamma_n: t-s=l\}}{|\Gamma_n|}\\
=&\frac{\sum\limits_{t\in  [-n,n)^d\cap  [-n-l,n-l)^d\cap\bZ^d}Y_tY_{t+l}}{| [-n,n)^d\cap [-n-l,n-l)^d\cap\bZ^d|}\\
&\cdot
\frac{|[-n,n)^d\cap \bZ^d|}{\sum\limits_{t\in [-n,n)^d \cap  \bZ^d}Y_t}\cdot\frac{| [-n,n)^d\cap  [-n-l,n-l)^d)\cap\bZ^d|}{|[-n,n)^d\cap\bZ^d|}.
\end{align*}
Letting $n$ go to infinity we obtain the assertion from the ergodic theorem for random fields (e.g. Lindenstrauss [\ref{Lindenstrauss}, Theorem 1.3]).
\end{proof}

%Later on we approximate our random field by a sequence of $m-$dependent random fields which satisfies a central limit theorem. For the reader's conveniance, we recall here the definition of $m-$dependence.
%\begin{definition}
%Let $X=(X_t)_{t\in\bR^d}$ be a random field. $X$ is called $m-$dependent with respect to the norm $\|\cdot\|$, if for every finite $A,B\subset\bR^d$ such that $dist_{\|\cdot\|}(A,B)\ge m$
%$\{X_t,\,t\in A\}$ and $\{X_s,\,s\in B\}$ are independent.
%\end{definition}
%\begin{remark} In the special case that all $a_l=1$ the sequence $(\Gamma_n)_{n\in \bN}$ is called a F{\o}lner sequence. They play an important role in ergodic theorems in the theory of amenable groups, see [\ref{Lindenstrauss}]. Under one additional condition on the sequence $(\Gamma_n)_{n\in\bN}$ it is possible to prove an pointwise ergodic theorem, see [\ref{Lindenstrauss}, Theorem 1.2, p. 260].
%\end{remark}
\begin{example}
Let $(Z_t)_{t\in \bZ^d}$ be a random field of independent and identically distributed random variables. A typical example of an ergodic random field is the moving average random field $M_t:=\sum\limits_{l\in  \bZ^d} a_l Z_{t-l}$, where $(a_l)_{l\in \bZ^d} \in \bR^{ \bZ^d}$ such that the sum is well-defined, i.e. the sum of the absolute values is almost surely finite. Let $\varphi:\bR\to \{0,1\}$ be a measurable function, then the random field $\varphi(M_t)$ is an ergodic and stationary random field. Assuming that $\varphi(M_t)>0$ with probability greater than $0$, $\varphi(M_t)$ satisfies the assumption of Lemma \ref{lemmaergodicsampling}.
\end{example}
\subsection{From Non-Random Sampling to Random Sampling} 
%Our sampling $(\Gamma_n)_{n\in\bN}\subset\bZ^d$ has certain properties, which are at first a little bit confusing.
We obtain a CLT on sequences $(\Gamma_n)_{n\in\bN}$ similar to the construction as in Lemma \ref{lemmaergodicsampling}  under the assumption that  $(Y_t)_{t\in\bZ^d}$ is \emph{$\alpha$-mixing}, which means that
%Let $(Z_t)_{t\in \Delta A\bZ^d}$ be a random field and $\mathcal{F}$ and $\mathcal{G}$ $\sigma$-algebras. We define
\begin{align*}
%\alpha(\mathcal{F},\mathcal{G})&:=\sup\{|P(A)P(B)-P(A\cap B)|\,|\,A\in\mathcal{F},B\in\mathcal{G}\},\\
\alpha_Y(k; u,v)&:=\sup\{\alpha(\sigma(Y_t,t\in A),\sigma(Y_t,t\in B))\,:\,dist(A,B)\ge k, |A|\le u, |B|\le v\}\to 0
\end{align*}
for $k\to\infty$ for every $u,v\in \bN$, where for two $\sigma$-fields $\mathcal{F}$ and $\mathcal{G}$, $\alpha(\mathcal{F},\mathcal{G})$ is defined by
\begin{align*}
\sup\{|P(A)P(B)-P(A\cap B)|\,:\,A\in\mathcal{F},B\in\mathcal{G}\}.
\end{align*}
A related but much stronger condition is $h$-dependence. A stationary random field $Y=(Y_t)_{t\in \bZ^d}$ or $Y=(Y_t)_{t\in \bR^d}$ is \emph{$h$-dependent} $(h>0)$, if for every two finite subsets $A,B\subset \bZ^d$ ($\subset \bR^d$, resp.) the two $\sigma$-fields $\sigma(Y_s: s\in A)$ and $\sigma(Y_s:s\in B)$ are independent if $dist(A,B)>h$.
\begin{theorem}\label{extra}
Let $(Y_t)_{t\in\bZ^d}$ be a $\{0,1\}-$valued $\alpha-$mixing random field, which is independent of the L\'{e}vy basis $L$ and satisfies $P(Y_0=1)>0$. Moreover, assume there exists a $\delta>0$ such that $Y$ satisfies
\begin{itemize}
\item[i)] for every $u,v\in\bN$ it holds $\alpha_Y(k;u,v)k^{d}\to 0$ for $k\to \infty$,
\item[ii)] for every $u,v\in\bN$ such that $u+v\le 4$ it holds $\sum\limits_{k=0}^\infty k^{d-1} \alpha_{Y}(k;u,v)<\infty$ and especially $\sum\limits_{k=0}^\infty k^{d-1}\alpha_Y(k;1,1)^{\delta/(2+\delta)}<\infty$.
\end{itemize}
Let \begin{align*}
\Gamma_n:=\{ t\in [-n,n)^d\cap \bZ^d \,:\, Y_t=1\}.
\end{align*} and $X=(X_t)_{t\in\bR^d}$ be a moving average random field with $X_t=\int_{\bR^d} f(t-u)\,dL(u)$ with $\bE |L([0,1]^d)|^{2+\delta}<\infty$ and $f\in L^1(\bR^d)\cap L^{2+\delta}(\bR^d)$. If
\begin{align*}
\sum\limits_{t\in \bZ^d}\bE Y_0Y_t\int\limits_{\bR^d}|f(-u)|\,|f(t-u)|\,\lambda^d(du)<\infty,
\end{align*}
then we have that
\begin{align*}
\frac{1}{\sqrt{|\Gamma_n|}}\sum\limits_{t\in \Gamma_n} \left(X_t- \beta\right)\stackrel{d}{\to} N\left(0,\sum\limits_{t\in \bZ^d}\frac{1}{\bE Y_0}\cov(Y_t(X_t-\beta),Y_0(X_0-\beta))\right),
\end{align*}
where $\beta=\bE L([0,1]^d) \int\limits_{\bR^d} f(u) \lambda^d(du)$.
In the special case that $Y$ is $h-$dependent for some finite $h>0$, it is enough to assume that $\bE |L([0,1]^d)|^{2}<\infty$ and $f\in L^1(\bR^d)\cap L^{2}(\bR^d)$.
\end{theorem}
\begin{example}
Every $h$-dependent random field $Y$ is $\alpha$-mixing with $\alpha_Y(k;u,v)=0$ for $|k|> h$. Other examples of (non-$h$-dependent) random fields $Y$ with suitable mixing rates can be constructed by [\ref{Doukhan} ,Theorem 2, p. 58].
\end{example}

\subsection{Non-Random Sampling of the Autocovariance}
Our object of interest is the estimator
\begin{align*}
\gamma^*_n(t):=\frac{1}{|\Gamma_n|}\sum\limits_{s\in \Gamma_n} X_sX_{s+t}
\end{align*}
for some $(\Gamma_n)_{n\in\bN}\subset \bZ^d$ of the autocovariance $\gamma_X(t)=\cov(X_0,X_t)$. We assume that $\Gamma_n$ satisfies the same conditions as in Theorem \ref{marc}.
We state a central limit theorem for the sample autocovariance which can be proven similar to Theorem \ref{marc}. Netherless, the  calculations are a little bit longer.\\
We assume that 
\begin{align}\label{eqrestrictions}
\bE\,L([0,1]^d)^4<\infty, \,\bE \,L([0,1]^d)=0, \,\sigma^2:=\bE\,L([0,1]^d)^2>0 
 \end{align}
and denote
\begin{align*} 
\eta:=\sigma^{-4}\bE \,L([0,1]^d)^4.
 \end{align*}
\begin{theorem}\label{sophie3}
Let $m\in\bN$ and $\Delta_1,\dotso,\Delta_m\in\bZ^d$, $\Gamma_n$ as in Theorem \ref{marc}, and let $(X_t)_{t\in\bR^d}=\left(\int_{\bR^d}f(t-s)dL(s)\right)_{t\in\bR^d}$ be a moving average random field such that it satisfies the assumptions (\ref{eqrestrictions}), $f\in L^2(\bR^d)\cap L^4(\bR^d)$ and
\begin{align*}
\sum\limits_{l\in \bZ^d} \int\limits_{\bR^d}  \sup\limits_{n\in\bN}a^{n}_l|f(u)f(u+l)f(u+\Delta_p)f(u+l+\Delta_d)|\lambda^d(du)<\infty
\end{align*}
for every $p,d\in \{1,\dotso,m\}$ and
\begin{align*}
\sum\limits_{l\in\bZ^d}\sup\limits_{n\in\bN}a^{n}_l\gamma_X(l)^2<\infty.
\end{align*}
Then
\begin{align}
\sqrt{|\Gamma_n|}(\gamma_n^{*}(\Delta_1)-\gamma_X(\Delta_1),\dotso,\gamma_n^{*}(\Delta_m)-\gamma_X(\Delta_m))\stackrel{d}{\to} N(0,V),
\end{align}
the multivariate normal distribution with mean $0$ and covariance matrix $V=(v_{pq})_{p,q\in \{1,\dotso,m\}}$ given by
\begin{align*}
v_{pq}=&\quad\,\sum\limits_{l\in\bZ^d}a_l\bigg( (\eta-3)\sigma^4 \int\limits_{\bR^d}f(u)f(u+\Delta_p)f(u+l)f(u+l+\Delta_q)\,\lambda^d(du)\\
&+\gamma_X(l)\gamma_X(l+\Delta_q-\Delta_p)+\gamma_X(l+\Delta_q)\gamma_X(l-\Delta_p)\bigg).
\end{align*}
\end{theorem}

\subsection{Random Sampling of the Autocovariance}
Now we present a theorem similiar to Theorem \ref{extra}.
\begin{theorem}\label{extra1}
Let $(Y_t)_{t\in \bZ^d}$ be a $\{0,1\}$-valued $\alpha-$mixing random field with mixing rates as in Theorem \ref{extra} ($\delta>0$), which is independent of the L\'{e}vy basis $L$. Let $X=(X_t)_{t\in\bR^d}$ be a moving average random field with $X_t=\int_{\bR^d} f(t-u)\,dL(u)$ such that  (\ref{eqrestrictions}) holds with $\bE |L([0,1]^d)|^{4+\delta}<\infty$ and $f\in L^2(\bR^d)\cap L^{4+\delta}(\bR^d)$. Let $\Delta_1,\dotso,\Delta_m\in \bZ^d$ and for every $p,d\in \{1,\dotso,m\}$ assume that
\begin{align*}
\sum\limits_{t\in \bZ^d}\bE Y_0Y_t\int\limits_{\bR^d}|f(u)f(u+t)f(u+\Delta_p)f(u+t+\Delta_d)|\lambda^d(du)<\infty
\end{align*}
and
\begin{align*}
\sum\limits_{l\in\bZ^d} \bE Y_0Y_l \gamma_X(l)^2<\infty.
\end{align*}
Then for $\Gamma_n:=\{ t\in [-n,n)^d\cap \bZ^d \,:\, Y_t=1\}$ we have
\begin{align}
\sqrt{|\Gamma_n|}(\gamma_n^{*}(\Delta_1)-\gamma_X(\Delta_1),\dotso,\gamma_n^{*}(\Delta_m)-\gamma_X(\Delta_m))\stackrel{d}{\to} N(0,V),
\end{align}
with covariance matrix $V=(v_{pq})_{p,q\in \{1,\dotso,m\}}$ given by
\begin{align}\label{komplizierte Kovarianz}
\nonumber v_{pq}=&\quad\,\sum\limits_{l\in\bZ^d}\frac{\bE Y_0Y_l}{\bE Y_0}\bigg( (\eta-3)\sigma^4 \int\limits_{\bR^d}f(u)f(u+\Delta_p)f(u+l)f(u+l+\Delta_q)\,\lambda^d(du)\\
&+\gamma_X(l)\gamma_X(l+\Delta_p-\Delta_q)+\gamma_X(l+\Delta_p)\gamma_X(l+\Delta_q)\bigg).
\end{align}
\end{theorem}
\section{Applications}\label{applications} \label{S4}
In this section we present an application of the previously stated theorems. We fix the dimension $d=3$ and estimate the parameter $\mu>0$ of the equation
\begin{align}\label{eqestimation}
(\mu-\Delta)X=dL,
\end{align}
where $L$ is a L\'{e}vy basis with $\bE L([0,1]^3)^2<\infty$.
%where $\Delta$ is the Laplace-operator, $dL$ denotes a L\'{e}vy basis with mean $\alpha>0$ and $X$ is the mild solution. The operator $\mu-\Delta$ is  special case of a pseudo-differential operator $\mathcal{L}=\sum\limits_{|\alpha|\le n}a_{\alpha}D^{\alpha}=p(D)$ such that $p(i\xi)$ has no zeroes in a strip \\$\{\xi\in\bC^d:\|\Im \xi\|\le \varepsilon+\delta\}$ for some $\varepsilon,\delta>0$ and
%\begin{align*}
%&\sup_{\eta\in \overline{B_{\varepsilon}(0)}}\left\|\frac{1}{p(i\cdot+\eta)}\right\|_{L^2}<\infty.
%\end{align*}
%At first let us show that the fundamental solution of such an partial differential operator $\mathcal{L}$ satisfies the assumptions from Theorem \ref{marc}.
%\begin{lemma}
%Let $G$ be the fundamental solution of $\mathcal{L}$, then it holds
%\begin{align*}
%&\sum\limits_{t\in\bZ^d}\int\limits_{\bR^d}|G(-u)|\,|G(t-u)|\,\lambda^d(du)<\infty.
%\end{align*}
%\end{lemma}
%\begin{proof}
%By [\ref{Reed}, Theorem XI.13, p.18] we see that
%\begin{align*}
%&\sum\limits_{t\in \bZ^d}\int\limits_{\bR^d}|G(-u)G(t-u)|\lambda^d(du)\\
%=&\sum\limits_{t\in \bZ^d}\int\limits_{\bR^d} |G(-u)G(t-u)|\exp\left(\varepsilon \|u\|+\varepsilon\|t-u\|\right) \exp\left(-\varepsilon \|u\|-\varepsilon\|t-u\|\right) \lambda^d(du)\\
%\le &\sum\limits_{t\in \bZ^d}\exp\left(-\varepsilon \|t\|\right)\|G\exp(\varepsilon\|\cdot\|)\|_{L^2}^2<\infty.
%\end{align*}
%\end{proof}
 The mild solution of $(\ref{eqestimation})$ can be written as
\begin{align}\label{mean}
X(x)=\int\limits_{\bR^d} G_\mu(x-z)dL(z),
\end{align}
 where $G_\mu(x):=\frac{\exp\left(-\sqrt{\mu}\|x\|\right)}{\|x\|}$ for $x\neq 0$, see [\ref{Dalang}, Definition 3.5] for the notion of the mild solution. That $G_{\mu}$ is a fundamental solution of $(\mu-\Delta)X=\delta_0$ follows e.g. from [\ref{Hofer}, Section 2.1, Equation (21)]. We see that $G_\mu \in L^1(\bR^3)\cap L^2(\bR^3)$, so $X$ exists since $\bE L([0,1]^3)^2<\infty$.\\
Calculating the mean we obtain
\begin{align*}
\bE X(x)=\bE X(0)=\bE L([0,1]^3)\int_{\bR^3}\frac{\exp\left(-\sqrt{\mu}\|x\|\right)}{\|x\|}dx=\frac{4\pi\bE L([0,1]^3)}{\mu},
\end{align*}
where the last equality follows by using spherical coordinates.
Our moment estimator is then given by
\begin{align}\label{estimator}
\widehat{\mu}_n=4\pi\bE L([0,1]^3)\frac{|\Gamma_n|}{\sum\limits_{k\in\Gamma_n} X(k)}.
\end{align}
%For the strong consistency we need that the sequence $\Gamma_n$ is a tempered F{\o}lner sequence, which means that
%\begin{align}
%\label{cond1}\lim\limits_{n\to\infty}\frac{(k+\Gamma_n)\Delta \Gamma_n}{|\Gamma_n|}&=0\textrm{ for all }k\in\Delta A\bZ^3 \textrm{ and }\\
%\label{cond2}\left|\bigcup_{k<n}(-\Gamma_k+\Gamma_n)\right|&\le C |\Gamma_n|\textrm{ for some constant }C>0,
%\end{align}
%where $\Delta$ denotes the symmetric difference. Condition $(\ref{cond1})$ is equivalent to $a_l^n\to 1$ for $n\to\infty$ for all $l\in\Delta A\bZ^3$, where $a_n^l$ is defined as in Theorem \ref{marc}.
\begin{corollary}
Let $\widehat{\mu}_n$ be defined as in $(\ref{estimator})$, $\bE L([0,1]^3)\neq 0$ and $\Gamma_n\subset \bZ^3$ satisfying the assumptions of Theorem \ref{marc}. Then $\widehat{\mu}_n$ defines a consistent and asymptotically normal estimator.
\end{corollary}
\begin{proof}
By Theorem \ref{marc} we conclude that $\widehat{\mu}_n^{-1}$ is asymptotically normal, as 
\begin{align*}
&\sum\limits_{t\in \bZ^d}\int\limits_{\bR^d}|G_\mu(-u)G_\mu(t-u)|\lambda^d(du)\\
=&\sum\limits_{t\in \ \bZ^d}\int\limits_{\bR^d} |G_\mu(-u)G_\mu(t-u)|\exp\left(\varepsilon \|u\|+\varepsilon\|t-u\|\right) \exp\left(-\varepsilon \|u\|-\varepsilon\|t-u\|\right) \lambda^d(du)\\
\le &\sum\limits_{t\in \bZ^d}\exp\left(-\varepsilon \|t\|\right)\|G_\mu\exp(\varepsilon\|\cdot\|)\|_{L^2}^2,
\end{align*}
which is finite for  $0<\varepsilon<\sqrt{\mu}$. Asymptotic normality and consistency of $\widehat{\mu}_n^{-1}$ implies consistency of $\widehat{\mu}_n$, and from both we obtain asymptotical normality of $\widehat{\mu}_n$.
\end{proof}
If in the situation above, additionally $\Gamma_n$ is a tempered F{\o}lner sequence, which means that
\begin{align}
\label{cond1}\lim\limits_{n\to\infty}\frac{((k+\Gamma_n)\setminus \Gamma_n) \cup (\Gamma_n\setminus (\Gamma_n+k))}{|\Gamma_n|}&=0\textrm{ for all }k\in\bZ^3 \textrm{ and }\\
\label{cond2}\left|\bigcup_{k<n}(-\Gamma_k+\Gamma_n)\right|&\le C |\Gamma_n|\textrm{ for some constant }C>0,
\end{align}
then the estimator $\hat{\mu}_n$ is strongly consistent by [\ref{Lindenstrauss}, Theorem 1.2, p. 260]. A simple example of a tempered F{\o}lner sequence  is $(-n,n]^d\cap \bZ^d$.
\section{Proof of Theorems \ref{marc} and \ref{extra}}\label{sectionproof1}
Since 
\begin{align*}
X_t=\int\limits_{\bR^d} f(t-u)dL'(u)+\bE (L([0,1]^d))\int\limits_{\bR^d}f(u)\lambda^d(du),
\end{align*}
where the mean zero L\'{e}vy basis $L'$ is defined by
\begin{align*}
L'(A):=L(A)-\bE L([0,1]^d)\lambda^d(A),\, A\in\mathcal{B}_b(\bR^d),
\end{align*}
and since 
\begin{align*}
\cov(Y_tX_t,Y_0X_0)=\cov(Y_t(X_t-\bE X_t),Y_0 (X_0-\bE X_0))
\end{align*}
in Theorem \ref{extra} by independence of $X$ and $Y$, we may and do assume for rest of this section that $\bE L([0,1]^d)=0$.

\begin{proof}[Proof of Theorem $\ref{marc}$]
For every $h\in\bN$ we define a new random field $(X^{(h)}_t)_{t\in\Delta A\bZ^d}$ by
\begin{align*}
X^{(h)}_t:=\int\limits_{\bR^d}f(t-u)\one_{[-h,h)^d}(t-u) \,dL(u).
\end{align*}
It is obvious that $(X_t^{(h)})_{t\in\Delta A\bZ^d}$ is $2\sqrt{d} h+1$-dependent.\\
We want to use [$\ref{Heinrich}$, Theorem 2, p. 135], which states the following: If we have a sequence $\{X_{nz}, z\in V_n\subset \bZ^d\}$, $n\in\bN$, of $m_n$-dependent  random fields $(m_n\ge 1)$ with $|V_n|\to \infty$, $\bE X_{nz}=0$ for all $z\in V_n$, $\bE \left(\sum_{z\in V_n} X_{nz}\right)^2=1$ and satisfying the conditions
\begin{align*}
&\sup_{n\in \bN}\sum\limits_{z\in V_n} \bE X_{nz}^2 <\infty\textrm{ and}\\
&m_n^{2d} \sum\limits_{z\in V_n} \bE X_{nz}^2\one_{|X_{nz}|\ge \varepsilon m_n^{-2d}}\to 0\textrm{ as }n\to\infty
\end{align*}
for every $\varepsilon>0$, then $\sum_{z\in V_n} X_{nz}\stackrel{d}{\to}N(0,1)$ as $n\to\infty$. In our case $m_n$ is constant, so the conditions are simpler.
We set $U^{(n,h)}_t:=\frac{1}{\sqrt{|\Gamma_n|}}X^{(h)}_t$. We calculate that
\begin{align}\label{irgendeinegleichung}
\bE \left( \sum\limits_{t\in \Gamma_n}U_t^{(n,h)} \right)^2
&=\frac{1}{|\Gamma_n|}\sum\limits_{t,s\in \Gamma_n}\bE X^{(h)}_tX_{s}^{(h)}&=\frac{1}{|\Gamma_n|}\sum\limits_{t,s\in \Gamma_n}\gamma_{X^{(h)}}(t-s)=\sum\limits_{l\in \bZ^d} a^{n}_l \gamma_{X^{(h)}}(l).
\end{align}
Letting $n$ go to infinity, we obtain by Lebesgue's dominated convergence theorem
\begin{align*}
\bE \left( \sum\limits_{t\in \Gamma_n}U_t^{(n,h)} \right)^2\to \sum\limits_{t\in\bZ^d}a_t\gamma_{X^{(h)}}(t).
\end{align*}
Furthermore, we immediately see that
\begin{align*}
\sum\limits_{t\in\Gamma_n} \bE (U^{(n,h)}_t)^2=\frac{1}{|\Gamma_n|}\sum\limits_{t\in \Gamma_n} \bE (X^{(h)}_t)^2=\gamma_{X^{(h)}}(0)<\infty
\end{align*}
and
\begin{align*}
\sum\limits_{t\in \Gamma_n}\bE \left((U_t^{(n,h)})^2\one_{|U_t^{(n,h)}|\ge \varepsilon}\right)&=\frac{1}{|\Gamma_n|}\sum\limits_{t\in \Gamma_n}\bE (X^{(h)}_t)^2\one_{|X^{(h)}_t|\ge \varepsilon \sqrt{|\Gamma_n|}}\\
&=\bE (X^{(h)}_0)^2\one_{|X^{(h)}_0|\ge \varepsilon \sqrt{|\Gamma_n|}}\to 0\qquad\textrm{ for }n\to\infty.
\end{align*}
Hence all conditions of [\ref{Heinrich}, Theorem 2, p. 135] as stated above are satisfied and we conclude that
\begin{align*}
\frac{1}{\sqrt{|\Gamma_n|}}\sum\limits_{t\in \Gamma_n}X^{(h)}_t\stackrel{d}{\to} Y^{(h)}
\end{align*}
for $n\to\infty$ with $Y^{(h)}\sim N(0,\sum_{t\in \bZ^d}a_t\gamma_{X^{(h)}}(t))$.\\
Observe that $\lim_{h\to\infty}\gamma_{X^{(h)}}(t)=\gamma_X(t)$ for all $t\in\bZ^d$ by (\ref{Ito}) and dominated convergence and $|\gamma_{X^{(h)}}(t)|\le \sigma^2 \int\limits_{\bR^d}|f(-u)|\,|f(t-u)|\lambda^d(du)$, hence we conclude by dominated convergence that
\begin{align*}
\lim\limits_{h\to\infty}\sum\limits_{t\in \bZ^d}a_t \gamma_{X^{(h)}}(t)=\sum\limits_{t\in \bZ^d}a_t \gamma_X(t)
\end{align*}
and hence
\begin{align*}
Y^{(h)}\stackrel{d}{\to} Y\sim N(0,\sum\limits_{t\in\bZ^d}a_t\gamma_X(t))\textrm{ for }h\to\infty.
\end{align*}
As in (\ref{irgendeinegleichung}), we obtain
\begin{align*}
&\bE \left(\frac{1}{\sqrt{|\Gamma_n|}}\sum\limits_{t\in \Gamma_n} (X_t-X_t^{(h)})\right)^2=\sum\limits_{l\in\bZ^d} a_l^n\gamma_{X-X^{(h)}}(l)\\
=&\sum\limits_{l\in\bZ^d} a_l^n\int\limits_{\bR^d}f(l-u)\one_{\bR^d\setminus [-h,h)^d} (t-u)f(-u)\one_{\bR^d\setminus [-h,h)^d}(-u)\lambda^d(du),
\end{align*}
hence
\begin{align*}
\lim\limits_{h\to\infty}\lim\limits_{n\to\infty}\bE \left(\frac{1}{\sqrt{|\Gamma_n|}}\left(\sum\limits_{t\in \Gamma_n}X_t-X_t^{(h)}\right)\right)^2=0
\end{align*}
from Lebesgue's dominated convergence theorem for series.
An application of Chebyshev's inequality gives for $\varepsilon>0$,
\begin{align*}
&\lim\limits_{h\to\infty}\lim\limits_{n\to\infty}P\left(\frac{1}{\sqrt{|\Gamma_n|}}\left|\sum\limits_{t\in \Gamma_n} X_t-X_t^{(h)}\right|>\varepsilon\right)=0.
\end{align*}
The claim then follows by a variant of Slutsky's theorem, e.g. [\ref{Brockwell}, Proposition 6.3.9, pp. 207-208].
\end{proof}
\begin{proof}[Proof of Theorem \ref{extra}]
The proof is very similiar to the proof of Theorem \ref{marc}. Let us start by approximating $X_t$ by $X_t^{(h)}$ as above. Observe that
\begin{align*}
\frac{1}{\sqrt{|\Gamma_n|}} \sum\limits_{t\in \Gamma_n} X^{(h)}_t=\frac{(2n)^{d/2}}{\sqrt{|\Gamma_n|}} \frac{1}{(2n)^{d/2}}\sum\limits_{t\in (-n,n]^d\cap \bZ^d} X^{(h)}_tY_t.
\end{align*}
We know that $\frac{(2n)^{d/2}}{\sqrt{|\Gamma_n|}}\to (\sqrt{\bE Y_0})^{-1}$, which follows from the ergodic theorem. Furthermore, as $(X^{(h)}_t)$ is $(2\sqrt{d} h+1)-$dependent and $Y$ is $\alpha-$mixing, we obtain that $(X^{(h)}_tY_t)_{t\in\bZ}$ is $\alpha-$mixing with the same rate as $Y$. From this and conditions i) and ii) of Theorem \ref{extra} we conclude by [\ref{Doukhan}, Theorem 3, p. 48] that
\begin{align*}
\frac{1}{(2n)^{d/2}} \sum\limits_{t\in \Gamma_n} X^{(h)}_t\stackrel{d}{\to}N\left(0,\sum\limits_{t\in \bZ^d} \frac{1}{\bE Y_0}\cov(X_t^{(h)}Y_t,X_0^{(h)}Y_0)\right)\textrm{ for }n\to\infty.
\end{align*}
Now by the same arguments as above we conclude that this theorem holds true when $Y$ is $\alpha$-mixing. When $Y$ is even $h'$-dependent for some $h'$, then $(X_t^{(h)}Y_t)_{t\in\bZ^d}$ is $\max\{h',2\sqrt{d} h+1\}$-dependent and we can use [\ref{Heinrich}, Theorem 2, p. 135] instead of [\ref{Doukhan}, Theorem 3, p. 48] and hence need weaker moment conditions.
\end{proof}
\section{Proof of Theorems \ref{sophie3} and \ref{extra1}}\label{sectionproof2}
\begin{proposition}\label{fourthmoment}
Let $f_1,\dotso,f_4 \in L^4(\bR^d)\cap L^2(\bR^d)$. It holds true that
\begin{align*}
\bE \prod_{i=1}^4 \int\limits_{\bR^d} f_i(t)dL(t)=&\quad(\eta-3)\sigma^4\int\limits_{\bR^d}f_1(u)f_2(u)f_3(u)f_4(u)\lambda^d(du)\\
&+\sigma^4\int\limits_{\bR^d} \prod_{i=1,2} f_i(u)\lambda^d(du)   \int\limits_{\bR^d} \prod_{i=3,4} f_i(u)\lambda^d(du)   \\
&+\sigma^4\int\limits_{\bR^d} \prod_{i=1,3} f_i(u)\lambda^d(du)   \int\limits_{\bR^d} \prod_{i=2,4} f_i(u)\lambda^d(du)   \\
&+\sigma^4\int\limits_{\bR^d} \prod_{i=1,4} f_i(u)\lambda^d(du)   \int\limits_{\bR^d} \prod_{i=2,3} f_i(u)\lambda^d(du).
\end{align*}
\end{proposition}
\begin{proof}
Follows directly from the proof of [\ref{Brandes}, Lemma 4.1].
\end{proof}

\begin{proposition}\label{es}Under the assumptions of Theorem \ref{sophie3}, for $\Delta_p,\Delta_q \in\bZ^d$, we have
\begin{align*}
|\Gamma_n|\cov (\gamma_n^{*}(\Delta_p),\gamma_n^{*}(\Delta_q))\to \sum\limits_{l\in\bZ^d} a_l T_l \quad\textrm{for }n\to\infty,
\end{align*}
where
\begin{align*}
T_l:=&(\eta-3)\sigma^4\int\limits_{\bR^d}f(u)f(u+l)f(u+\Delta_p)f(u+l+\Delta_q)\lambda^d(du)\\
&+\gamma_X(l)\gamma_X(l+\Delta_q-\Delta_p)+\gamma_X(l+\Delta_q)\gamma_X(l-\Delta_p).
\end{align*}
\end{proposition}
\begin{proof}
A direct calculation gives us
\begin{align*}
|\Gamma_n|\cov (\gamma_n^{*}(\Delta_p),\gamma_n^{*}(\Delta_q))&=\frac{1}{|\Gamma_n|}\sum\limits_{s,t\in \Gamma_n}\cov(X_tX_{t+\Delta_p},X_sX_{s+\Delta_q})\\
&=\frac{1}{|\Gamma_n|}\sum\limits_{s,t\in \Gamma_n}\bE (X_tX_sX_{t+\Delta_p}X_{s+\Delta_q})-\gamma_X(\Delta_p)\gamma_X(\Delta_q)\\
&=\frac{1}{|\Gamma_n|}\sum\limits_{s,t\in \Gamma_n}\bE (X_0X_{s-t}X_{\Delta_p}X_{s-t+\Delta_q})-\gamma_X(\Delta_p)\gamma_X(\Delta_q)\\
&=\frac{1}{|\Gamma_n|}\sum\limits_{s,t\in \Gamma_n}T_{s-t},
\end{align*}
which follows from Proposition $\ref{fourthmoment}$, and we get that
\begin{align*}
\frac{1}{|\Gamma_n|}\sum\limits_{s,t\in \Gamma_n}T_{s-t}=\sum\limits_{l\in\bZ^d}a^{n}_lT_l.
\end{align*}
 By our assumptions and Lebesgue's dominated convergence theorem for series we conclude that
\begin{align*}
|\Gamma_n|\cov (\gamma_n^{*}(\Delta_p),\gamma_n^{*}(\Delta_q))\to \sum\limits_{l\in\bZ^d} a_l T_l \quad\textrm{for }n\to\infty.
\end{align*}
\end{proof}
\begin{proof}[Proof of Theorem \ref{sophie3}]
Let $h\in \bN$ and $X^{(h)}_t$ be given by
\begin{align*}
X^{(h)}_t:=\int\limits_{\bR^d} f^{(h)}(t-u)\,dL(u),
\end{align*}
where $f^{(h)}(u):=f(u)\one_{[-h,h)^d}(u)$. We define
\begin{align*}
U^{(h)}_t:=(X^{(h)}_tX_{t+\Delta_1}^{(h)},\dotso,X^{(h)}_tX^{(h)}_{t+\Delta_m}).
\end{align*}
Now observe that $(U^{(h)}_t)_{t\in\bZ^d}$ is $(2\sqrt{d}h+2\sup_{i=1,\dotso,m}\|\Delta_i\|+1)$-dependent. We want to show that
\begin{align}\label{kk}
\frac{1}{\sqrt{|\Gamma_n|}}\sum\limits_{t\in \Gamma_n}(U_t^{(h)}-(\gamma_{X^{(h)}}(\Delta_1),\dotso,\gamma_{X^{(h)}}(\Delta_m)))\stackrel{d}{\to}Y^{(h)}\stackrel{d}{=}N(0,V^{(h)})
\end{align}
as $n\to\infty$, where $V^{(h)}=(v^{(h)}_{pq})_{p,q\in\{1,\dotso,n\}}$ is defined by $(\ref{komplizierte Kovarianz})$ with $f$ replaced by $f^{(h)}$. 
Let $\alpha=(\alpha_1,\dotso,\alpha_m)\in\bR^m\setminus\{0\}$.  Define $K^{(h)}_t:=\alpha(U^{(h)}_t-(\gamma_{X^{(h)}}(\Delta_1),\dotso,\gamma_{X^{(h)}}(\Delta_m)))^T$, which is also $(2\sqrt{d}h+2\sup_{i=1,\dotso,m}\|\Delta_i\|+1)$-dependent. Then we see that $\bE K^{(h)}_t=0$ and
\begin{align*}
\frac{1}{|\Gamma_n|}\bE \left(\sum\limits_{t\in \Gamma_n} K^{(h)}_t\right)^2&=\frac{1}{|\Gamma_n|}\sum\limits_{t,s\in \Gamma_n}\bE\,K^{(h)}_tK^{(h)}_s\\
&=\frac{1}{|\Gamma_n|}\sum\limits_{t,s\in \Gamma_n}\bE (\alpha (U^{(h)}_t-(\gamma_{X^{(h)}}(\Delta_1),\dotso,\gamma_{X^{(h)}}(\Delta_m)))^T\\&\qquad\qquad\qquad\alpha ( (U^{(h)}_s-(\gamma_{X^{(h)}}(\Delta_1),\dotso,\gamma_{X^{(h)}}(\Delta_m)))^T)\\
&=\frac{1}{|\Gamma_n|}\sum\limits_{t,s\in \Gamma_n}\bE\sum\limits_{i,j=1}^m \alpha_i\alpha_j (X^{(h)}_tX^{(h)}_{t+\Delta_i}-\gamma_{X^{(h)}}(\Delta_i))(X^{(h)}_sX^{(h)}_{s+\Delta_j}-\gamma_{X^{(h)}}(\Delta_j))\\
&=\frac{1}{|\Gamma_n|}\sum\limits_{t,s\in \Gamma_n}\sum\limits_{i,j=1}^m \alpha_i\alpha_j \cov(X^{(h)}_tX^{(h)}_{t+\Delta_i},X^{(h)}_sX_{s+\Delta_j}^{(h)}).
\end{align*}
By Proposition $\ref{es}$ we conclude that
\begin{align*}
\frac{1}{|\Gamma_n|}\bE \left(\sum\limits_{t\in \Gamma_n} K^{(h)}_t\right)^2\to \sum\limits_{i,j=1}^m  \alpha_i\alpha_j v^{(h)}_{ij}
\end{align*}
for $n\to \infty$. Furthermore, for every $\varepsilon>0$ we have
\begin{align*}
&\lim\limits_{n\to\infty} \frac{1}{|\Gamma_n|}\sum\limits_{t\in \Gamma_n}\bE (K^{(h)}_t)^2\one_{|K^{(h)}_t|\ge |\Gamma_n| \varepsilon}\\
=&\lim\limits_{n\to\infty}\bE (K^{(h)}_0)^2\one_{|K^{(h)}_0|\ge |\Gamma_n| \varepsilon}=0
\end{align*}
and
\begin{align*}
\frac{1}{|\Gamma_n|}\sum\limits_{t\in \Gamma_n}\bE (K^{(h)}_t)^2=\bE (K^{(h)}_0)^2<\infty.
\end{align*}
By  [$\ref{Heinrich}$, Theorem 2, p. 135] we conclude that
\begin{align*}
\frac{1}{\sqrt{|\Gamma_n|}}\sum\limits_{t\in \Gamma_n}K^{(h)}_t\stackrel{d}{\to}N(0,\sum\limits_{i,j=1}^m\alpha_i\alpha_j v^{(h)}_{ij}),\, n\to\infty.
\end{align*}
By the Cr\'{a}mer-Wold Theorem we see that $(\ref{kk})$ holds true. Next we have to show that $V^{(h)}\to V$ for $h\to\infty$. But this follows from dominated convergence, since $f^{(h)}\to f$ in $L^4(\bR^d)$ and in $L^2(\bR^d)$ as $h\to\infty$, since $|f^{(h)}|\le |g|$ and by $(\ref{Ito})$.
Hence we get
\begin{align*}
Y^{(h)}\stackrel{d}{\to} Y\sim N(0,V)\textrm{ as }h\to\infty.
\end{align*}
The claim will now follow by  [\ref{Brockwell}, Proposition 6.3.9, pp. 207-208] if we can show that for any $\varepsilon>0$,
\begin{align}\label{0}
\lim\limits_{h\to\infty}\lim\limits_{n\to\infty}P\left(\sqrt{|\Gamma_n|}\left|\gamma_n^{*}(\Delta_i)-\gamma_X(\Delta_i)-\frac{1}{|\Gamma_n|}\sum\limits_{t\in \Gamma_n}X^{(h)}_tX^{(h)}_{t+\Delta_i}+\gamma_{X^{(h)}}(\Delta_i)\right|>\varepsilon\right)=0.
\end{align}
This follows by showing that
\begin{align}\label{0}
\lim\limits_{h\to\infty}\lim\limits_{n\to\infty}\mathbb{E}{|\Gamma_n|}\left|\gamma_n^{*}(\Delta_i)-\gamma_X(\Delta_i)-\frac{1}{|\Gamma_n|}\sum\limits_{t\in \Gamma_n}X^{(h)}_tX^{(h)}_{t+\Delta_i}+\gamma_{X^{(h)}}(\Delta_i)\right|^2=0.
\end{align}
as an application of the Dominated convergence Theorem similar to the end of the proof of [\ref{Lindner}, Theorem 3.5, p. 1302] and therefore we obtain our desired result.
\end{proof}
\begin{proof}[Proof of Theorem \ref{extra1}]
We observe that $$\sum_{t\in \Gamma_n}(X_tX_{t+\Delta_i}-\gamma_X(\Delta_i))=\sum_{t\in [-n,n)^d\cap  \bZ^d}Y_{t}(X_tX_{t+\Delta_i}-\gamma_X(\Delta_i))$$ and 
\begin{align*}
&\cov(Y_t(X_t^{(h)}X_{t+\Delta_i}^{(h)}-\gamma_{X^{(h)}}(\Delta_i)),Y_s (X_s^{(h)}X_{s+\Delta_j}^{(h)}-\gamma_{X^{(h)}}(\Delta_j)))\\
=& \bE Y_t(X_t^{(h)}X_{t+\Delta_i}^{(h)}-\gamma_{X^{(h)}}(\Delta_i))Y_s (X_s^{(h)}X_{s+\Delta_j}^{(h)}-\gamma_{X^{(h)}}(\Delta_j)\\
&-\bE Y_t(X_t^{(h)}X_{t+\Delta_i}^{(h)}-\gamma_{X^{(h)}}(\Delta_i))\bE Y_s (X_s^{(h)}X_{s+\Delta_j}^{(h)}-\gamma_{X^{(h)}}(\Delta_j))\\
=&\bE Y_tY_s\bE(X_t^{(h)}X_{t+\Delta_i}^{(h)}-\gamma_{X^{(h)}}(\Delta_i))(X_s^{(h)}X_{s+\Delta_j}^{(h)}-\gamma_{X^{(h)}}(\Delta_j).
\end{align*}
Repeating the same steps as in the proof of Theorem \ref{extra} gives the claim.
\end{proof}
\section*{Acknowledgement:} Partial support by DFG grant LI 1026/6-1 is gratefully acknowledged. The author would like to thank Alexander Lindner for introducing him to this topic, giving him the opportunity to work on it and for many interesting and fruitful discussions. Furthermore, the author would like to thank the editor and the anonymous referee for their careful reading and valuable comments which improved the exposition of the paper.

\vspace{1cm}
David Berger\\
 Ulm University, Institute of Mathematical Finance, Helmholtzstra{\ss}e 18, 89081 Ulm,
Germany\\
email: david.berger@uni-ulm.de
\end{document}